\documentclass[12pt]{amsart}

\usepackage{tikz}
\usetikzlibrary{calc, decorations.pathreplacing, arrows.meta}

\usepackage[colorlinks=true,urlcolor=blue, citecolor=red,linkcolor=blue,linktocpage,pdfpagelabels, bookmarksnumbered,bookmarksopen]{hyperref}
\usepackage[hyperpageref]{backref}
\usepackage{cleveref}
\usepackage{xcolor}
\usepackage{amsthm} 
\usepackage{latexsym,amsmath,amssymb}
\usepackage{accents}
\usepackage{a4wide}
\usepackage{soul}
\usepackage{mathtools} 
\usepackage{xparse} 
\usepackage{enumitem}		

\usepackage{calc}

\usepackage{accents}

\definecolor{indigo}{rgb}{0.29, 0.0, 0.51}
\definecolor{p1}{gray}{0.4}
\definecolor{p2}{gray}{0.6}
\definecolor{p3}{gray}{0.98}
\definecolor{p4}{gray}{0.8}
\definecolor{p5}{gray}{0.9}

plus 6pt minus 12pt
\newcommand{\sing}{\operatorname{sing}}

\def\eps{\varepsilon}

\def\id{{\rm id\, }}

\newcommand{\m}{\mathcal{M}}

\def\n{{\mathcal N}}

\def\n{{\mathcal{M}}}
\def\n{{\mathcal N}}

\def\S{{\mathbb S}}

\newtheorem{theorem}{Theorem}

\newtheorem{corollary}[theorem]{Corollary}
\newtheorem{proposition}[theorem]{Proposition}

\newtheorem{openproblem}[theorem]{Open Problem}

\def\dist{{\rm dist\,}}

\newcommand{\dif}{\,\mathrm{d}}
\newcommand{\dx}{\dif x}

\newcommand{\R}{\mathbb{R}}

\newcommand{\brac}[1]{\left (#1 \right )}
\newcommand{\abs}[1]{\left |#1 \right |}

\newcommand{\norm}[1]{\left\|{#1}\right\|}

\newcommand{\barint}{
\rule[.036in]{.12in}{.009in}\kern-.16in \displaystyle\int }

\newcommand{\barcal}{\mbox{$ \rule[.036in]{.11in}{.007in}\kern-.128in\int $}}

\def\mvint_#1{\mathchoice
          {\mathop{\vrule width 6pt height 3 pt depth -2.5pt
                  \kern -8pt \intop}\nolimits_{\kern -3pt #1}}%
          {\mathop{\vrule width 5pt height 3 pt depth -2.6pt
                  \kern -6pt \intop}\nolimits_{#1}}%
          {\mathop{\vrule width 5pt height 3 pt depth -2.6pt
                  \kern -6pt \intop}\nolimits_{#1}}%
          {\mathop{\vrule width 5pt height 3 pt depth -2.6pt
                  \kern -6pt \intop}\nolimits_{#1}}}

\numberwithin{theorem}{section} \numberwithin{equation}{section}

\newcommand{\aleq}{\precsim}

\def\XXint#1#2#3{{\setbox0=\hbox{$#1{#2#3}{\int}$}
     \vcenter{\hbox{$#2#3$}}\kern-.5\wd0}}

\let\latexchi\chi
\makeatletter
\renewcommand\chi{\@ifnextchar_\sub@chi\latexchi}
\newcommand{\sub@chi}[2]{
  \@ifnextchar^{\subsup@chi{#2}}{\latexchi^{}_{#2}}%
}
\newcommand{\subsup@chi}[3]{%
  \latexchi_{#1}^{#3}%
}
\makeatother

\DeclarePairedDelimiterX{\FromTo}[2]{\lparen}{\rparen}{\ifblank{#2}{#1}{#1;#2}}

\title[Non-uniqueness for minimizing harmonic maps]{Generic non-uniqueness of minimizing harmonic maps from a ball to a sphere}

\author{Antoine Detaille}
\address[Antoine Detaille]{
Universite Claude Bernard Lyon 1, ICJ UMR5208, CNRS, Ecole Centrale de Lyon, INSA Lyon, Université Jean Monnet, 
69622 Villeurbanne, 
France.}
\email{antoine.detaille@univ-lyon1.fr }

\author{Katarzyna Mazowiecka}
\address[Katarzyna Mazowiecka]{
Institute of Mathematics,%
University of Warsaw,
Banacha 2,
02-097 Warszawa, Poland}
\email{k.mazowiecka@mimuw.edu.pl}

\begin{document}

 \begin{abstract}
In this note, we study non-uniqueness for minimizing harmonic maps from $B^3$ to $\S^2$. We show that every boundary map can be modified to a boundary map that admits multiple minimizers of the Dirichlet energy by a small $W^{1,p}$-change for $p<2$.
This strengthens a remark by the second-named author and Strzelecki.
The main novel ingredient is a homotopy construction, which is the answer to an easier variant of a challenging question regarding the existence of a norm control for homotopies between \( W^{1,p} \) maps.

 \end{abstract}

\sloppy

\maketitle
\tableofcontents
\sloppy
\section{Introduction}
Minimizing harmonic maps from $B^3$ to $\S^2$ are defined as mappings with the least Dirichlet energy
\begin{equation}\label{eq:energy}
 E(u)\coloneqq \int_{B^3} |\nabla u|^2 \dx
\end{equation}
among maps $u\in W^{1,2}(B^3,\S^2)$ with fixed boundary datum $u\big\rvert_{\partial B^3} = \varphi\in W^{\frac 12,2}(\partial B^3,\S^2)$. 
Here, we minimize in the class of Sobolev maps with values in a manifold (in our case, a sphere); for $s>0$ and $p\ge1$, this space is defined as
\[
 W^{s,p}(\m,\n) \coloneqq \{ v\in W^{s,p}(\m,\R^L)\colon v(x) \in \n \text{ for a.e. } x\in \m\}\text{,}
\]
where $\n\subset \R^L$ is a Riemannian manifold embedded into $\R^L$ (in our case, $\n = \S^2$) and $\m$ is a compact Riemannian manifold (in our case, $\m = B^3$ or $\m = \S^2$).

The space \( W^{1,2}(B^3,\S^2) \) is not a linear space, but it is nevertheless a complete metric space endowed with the metric defined by
\[
	\dist{(u,v)} = \norm{u-v}_{W^{1,2}(B^{3})}\text{.}
\]
We emphasize that, although being a subset of it, the class \( W^{1,2}(B^3,\S^2) \) exhibits some striking qualitative differences with the linear space \(  W^{1,2}(B^3,\R^3) \).
For example, not every mapping $u \in W^{1,2}(B^3,\S^2)$ can be approximated by smooth maps $u_i\in C^\infty(B^3,\S^2)$ in the strong topology of $W^{1,2}$; see \cite[Section 4]{SU2}. However, maps $\varphi \in W^{1,2}(\S^2,\S^2)$ can be approximated in $W^{1,2}$ by smooth maps $\varphi_{i} \in C^\infty(\S^2,\S^2)$; see~\cite[Section 3]{SU1}.

For $\varphi \in W^{\frac12,2}(\partial B^3, \S^2)$, we also define the space
\[
  W^{1,2}_\varphi(B^3,\S^2) \coloneqq \{ v\in W^{1,2}(B^3,\S^2)\colon  v= \varphi \text{ on } \partial B^3 \text{ in the trace sense}\}
\]
and note that this space is always nonempty.
For instance, for a given smooth boundary datum $\varphi\in C^\infty(\partial B^3,\S^2)$, one can easily construct an extension $u\in W^{1,2}(B^3,\S^2)$ of \( \varphi \), simply by considering $u(x) = \varphi\bigl(\frac{x}{|x|}\bigr)$. 
More generally, any boundary map $\varphi \in W^{\frac 12,2}(\partial B^3,\S^2)$ admits an extension $u\in W^{1,2}(B^3,\S^2)$; see~\cite[Theorem 6.2]{HLp}.
Once again, we emphasize that this is not an immediate consequence of the analogue property of linear Sobolev spaces.
For example, there exists a boundary datum $\varphi \in W^{\frac 12,2}(\partial B^3,\S^1)$ which has \emph{no} extension $u\in W^{1,2}(B^3,\S^1)$; see~\cite[6.3]{HLp}.

Minimizing harmonic maps satisfy the following system of Euler--Lagrange  equations
\begin{equation}\label{eq:EL}
\left\{
\begin{array}{rcll}
-\Delta u & = & |\nabla u|^2 u\qquad&\mbox{in $B^3$,}\\
u & = &\varphi\qquad&\mbox{on $\partial B^3$}\text{.}
\end{array}
\right.
\end{equation}
It is known that for every non-constant boundary datum, the system \eqref{eq:EL} admits infinitely many solutions; see \cite{RivierePhd}. Minimizers of \eqref{eq:energy} are not the only solutions to \eqref{eq:EL} (see, e.g., \cite[Section 3]{HKLu}). However, even in the class of minimizing harmonic maps, we do not have uniqueness for a given boundary datum $\varphi\colon B^3 \to \S^2$; there are many known examples.
To list a few:
\begin{itemize}
 \item in \cite[Section 3]{HardtKinderlehrer}, there is an example of a planar boundary datum which admits two different minimizers, one with values on the southern hemisphere and the other one with values on the northern hemisphere;
 \item in \cite[2.2. Corollary]{HKL-thevariety88}, there is an example of a boundary datum for which there exists a 1-parameter family of distinct energy minimizing maps;
 \item in \cite[Section 5]{HardtLin89}, there is an example of a boundary map which serves as a boundary datum for at least two minimizers, one singular and the other one regular;
 \item in \cite[5.5 Theorem]{AL88}, there is an example of a boundary datum with mirror symmetry for which there are at least two different minimizers without the mirror symmetry.
\end{itemize}

Nevertheless, in the class of minimizing harmonic maps, we have the following \emph{generic uniqueness} result (\cite{AL88} attributes this theorem to Almgren).
\begin{theorem}[{\cite[Theorem 4.1]{AL88}}]\label{th:ALuniqueness}
 Let $\varphi \in W^{1,2}(\S^2,\S^2)$. For every $\eps>0$, there exists $\psi\in W^{1,2}(\S^2,\S^2)$ such that $\norm{\varphi - \psi}_{W^{1,2}(\S^2)} < \eps$ and for which there exists exactly one energy minimizer $u\colon B^3 \to \S^2$ having boundary datum $\psi$. Moreover, $\psi$ coincides with $\varphi$ outside of $B_\eps(x)\cap \S^2$, for some $x\in \S^2$.
\end{theorem}
In \cite{MazowieckaStrzelecki}, the second-named author and Strzelecki suspected that \emph{generic non-uniqueness} occurs, when taking into account small perturbation of the boundary datum in the topology of the space $W^{1,p}$ for $p<2$. The main result of this note is the strengthening of \cite[Remark 4.1]{MazowieckaStrzelecki}.
\begin{theorem}\label{th:main}
Let $\varphi \in C^\infty(\S^2,\S^2)$. For every $\eps>0$, there exists $\psi\in C^\infty(\S^2,\S^2)$ such that $\norm{\varphi-\psi}_{W^{1,p}(\S^2,\S^2)}<\eps$ which serves as a boundary datum for at least two energy minimizing maps from $B^3$ to $\S^2$ having a different number of singularities.
\end{theorem}

Otherwise stated, \Cref{th:main} asserts that boundary data for which non-uniqueness occurs are dense in \( W^{1,p}(\S^{2},\S^{2}) \).
This strengthens~\cite[Section 5]{HardtLin89} and \cite[Remark~4.1]{MazowieckaStrzelecki}, which provide existence of \emph{one} boundary map for which non-uniqueness occurs.
To be precise, as it is stated, Theorem~\ref{th:main} only asserts that boundary data subjected to non-uniqueness are dense in \( C^{\infty}(\S^{2},\S^{2}) \) with respect to the \( W^{1,p} \) topology.
In turn, \( C^{\infty}(\S^{2},\S^{2}) \) is dense in \( W^{1,p}(\S^{2},\S^{2}) \) (see e.g.~\cite[Theorem~1]{BethuelZheng1988}), which ensures the density of boundary data for which non-uniqueness occurs in the whole \( W^{1,p}(\S^{2},\S^{2}) \).

Both \Cref{th:ALuniqueness} and \Cref{th:main} are in line with the \emph{stability} results: 
On one hand, it is known that small perturbations of boundary data (for which there is a unique minimizer) in the $W^{1,2}$ norm do not change the number of singularities for corresponding minimizers (see \cite{HardtLin89} for perturbations in the $W^{1,\infty}$ norm, \cite{MMS} and \cite{SiranLi} for perturbations in the $W^{1,2}$ norm). 
On the other hand, small perturbations of the boundary datum in the $W^{1,p}$ norm for $p<2$ can change the number of singularities for corresponding minimizers \cite{MazowieckaStrzelecki}.

We prove \Cref{th:main} in \Cref{s:proof}. 
To do so, roughly speaking, we follow an example by Hardt--Lin \cite[Section 5]{HardtLin89}. 
We start with any smooth boundary datum and use the construction of a boundary map (homotopic to the original one) of \cite{MazowieckaStrzelecki} (see \cite{MazowieckaPhD} for necessary modifications) for which a \emph{Lavrentiev gap phenomenon} occurs. 
In \Cref{s:homotopy}, we show that a homotopy between these two maps can be chosen small in $W^{1,p}$-norm for $p<2$, which is the novelty of this note, and prove that within this homotopy, there is a boundary datum with the required properties.

As we explained, our key contribution in this note, which allows the transition from the existence to the density of boundary data where non-uniqueness occurs, is the homotopy construction presented in \Cref{s:homotopy}.
We conclude this introduction with some extra comments concerning this construction.

Assume that one is given \( 1 \leq p < 2 \) and two maps \( \varphi \) and \( \psi \in C^{\infty}(\S^{2},\S^{2}) \) that have the same topological degree.
Therefore, there exists a continuous, and even smooth homotopy connecting \( \varphi \) to \( \psi \).
A natural question is whether or not, knowing that \( \varphi \) and \( \psi \) are close with respect to the \( W^{1,p} \) distance, one can choose the homotopy between \( \varphi \) and \( \psi \) to remain close to \( \varphi \) and \( \psi \) all along the deformation.
More precisely, one could for instance expect that there exists a constant \( C > 0 \) depending on \( p \) such that a homotopy \( H \in C^{\infty}(\S^{2} \times [0,1],\S^{2}) \) between \( \varphi \) and \( \psi \) can be chosen so that
\begin{equation}
\label{eq:uniform_control_homotopy}
	\norm{\varphi - H_{t}}_{W^{1,p}(\S^{2})} \leq C\norm{\varphi - \psi}_{W^{1,p}(\S^{2})}
	\quad
	\text{for every \( 0 \leq t \leq 1\).}
\end{equation}
Here, \( H_{t} \) stands for the map \( H(\cdot,t) \).
The question is already interesting if we assume in addition that \( \varphi \) and \( \psi \) coincide outside of a small disk.
For instance, one could ask whether or not a homotopy such that~\eqref{eq:uniform_control_homotopy} holds can be found under the additional assumption that \( \varphi = \psi \) outside of a ball of radius \( r \), for some \( r > 0 \) sufficiently small, possibly depending on the map \( \varphi \) that would be fixed in advance.

We are not able to solve this question, and a precise statement of the problem in a more general context is given as~\Cref{open_problem:small_homotopy}.
However, we are able to solve a weaker version of this problem, which is nevertheless sufficient for our purposes.
Namely, we prove that, if the maps \( \varphi \) and \( \psi \) coincide outside of a small ball, then a smooth homotopy between them can be found such that \( \norm{\varphi - H_{t}}_{W^{1,p}(\S^{2})} \) is controlled, not by the distance between \( \varphi \) and \( \psi \), but by the sum of their norms on a neighborhood of the region where they differ.
This is the content of the main result of~\Cref{s:homotopy},~\Cref{proposition:controlled_homotopies}.
This allows us to deduce that, for a fixed \( \varphi \) and a given \( \varepsilon > 0 \), one can choose the radius \( r > 0 \) sufficiently small such that, for any map \( \psi \) sufficiently close to \( \varphi \) such that \( \varphi = \psi \) outside of \( B_{r}(x) \), a homotopy \( H \) connecting \( \varphi \) to \( \psi \) can be found such that
\[
	\norm{\varphi - H_{t}}_{W^{1,p}(\S^{2})}
	\leq
	\varepsilon
	\quad
	\text{for every \( 0 \leq t \leq 1 \);}
\]
see~\Cref{corollary:small_homotopies}.
This is sufficient to prove our main result,~\Cref{th:main}, but does not solve~\Cref{open_problem:small_homotopy}, as in our proof the radius \( r > 0 \) of the ball outside of which the maps \( \varphi \) and \( \psi \) are required to coincide has to depend on \( \varepsilon \), ruling out the possibility of controlling \( \norm{\varphi - H_{t}}_{W^{1,p}(\S^{2})} \) uniformly in \( t \) solely by \( \norm{\varphi - \psi}_{W^{1,p}(\S^{2})} \) with our argument.

{\bf Notation.} We denote by $B^3$ a Euclidean unit ball in $\R^3$. We will write $\S^n$ for the unit $n$-dimensional sphere. For a point $x\in \S^n$ and $r>0$, we will write $B_r(x)$ for a geodesic ball of radius $r$ around $x$. We will write $A \aleq B$ whenever there is a constant $C$ (independent of all crucial quantities) such that $A\le C B$. Throughout this paper, the term \emph{minimizer} will always refer to an $\S^2$-valued mapping minimizing the Dirichlet energy with given boundary datum.

{\bf Acknowledgments.}
\begin{itemize}
 \item The project is co-financed by the Polish National Agency for Academic Exchange within Polish Returns Programme - BPN/PPO/2021/1/00019/U/00001 (KM).
 \item The project is co-financed by National Science Centre, Poland grant 2022/01/1/ST1/00021 (KM).
 \item The project was initiated while AD was visiting the Institute of Mathematics of the University of Warsaw. He really thanks them for their hospitality.
\end{itemize}

\section{Homotopy construction}\label{s:homotopy}
We will assume in this section that $\n$ is a (non necessarily compact) Riemannian manifold.
We work on the sphere $ \S^n $, but the result may be readily extended to an arbitrary domain, either an open subset of $ \R^n $ or a Riemannian manifold $ \m $ of dimension $ n $.
We also always assume that $ p < n $.

\begin{proposition}
\label{proposition:controlled_homotopies}
	Let $ \varphi \in C^{\infty}(\S^n,\n)$ and $p<n$.
	For every $ r > 0 $, for every $ x \in \S^n $, and every $ \psi \in C^{\infty}(\S^n,\n)$ homotopic to $ \varphi $ and satisfying $ \varphi = \psi $ on $ \S^n \setminus B_r(x)  $, there exists a homotopy $ H \in C^{\infty}(\S^n \times [0,1],\n)$ from $ \varphi $ to $ \psi $ such that
	\[
		\sup_{0 \leq t \leq 1} \norm{\varphi - H_{t}}_{W^{1,p}(\S^n)}
		\leq
		C\brac{\norm{\varphi}_{W^{1,p}(B_{2r}(x))} + \norm{\psi}_{W^{1,p}(B_{2r}(x))}}\text{,}
	\]
	for some constant $ C > 0 $ depending only on $ n $ and $ p $.
\end{proposition}

This proposition can be used in combination with Lebesgue's lemma to obtain a homotopy which remains close to $ \varphi $ in $ W^{1,p} $.
Indeed, choosing $ r $ sufficiently small, depending on $ \varphi $, we may ensure that $ \norm{\varphi}_{W^{1,p}(B_{2r}(x))} $ is as small as we want, uniformly with respect to $ r $.
Since $ \norm{\psi}_{W^{1,p}(B_{2r}(x))} \leq \norm{\varphi}_{W^{1,p}(B_{2r}(x))} + \norm{\varphi-\psi}_{W^{1,p}(\S^n)} $, assuming in addition that $ \norm{\varphi-\psi}_{W^{1,p}(\S^n)} $ is small, we can make $ \sup_{0 \leq t \leq 1} \norm{\varphi - H_{t}}_{W^{1,p}(\S^n)} $ as small as we want. 
This yields the following corollary.

\begin{corollary}
\label{corollary:small_homotopies}
	Let $ \varphi \in C^{\infty}(\S^n,\n) $ and $p<n$.
	For every $ \varepsilon > 0 $, there exists $ r > 0 $ sufficiently small, depending on $ \varphi $, and there exists $ \delta > 0 $ such that, for every $ x \in \S^n $ and every $ \psi \in C^{\infty}(\S^n,\n) $ homotopic to $ \varphi $ and satisfying $ \varphi = \psi $ on $ \S^n \setminus B_r(x)  $ and $ \norm{\varphi-\psi}_{W^{1,p}(\S^n)} \leq \delta $, there exists a homotopy $ H \in C^{\infty}(\S^n\times[0,1],\n)$ from $ \varphi $ to $ \psi $ such that
	\[
		\sup_{0 \leq t \leq 1} \norm{\varphi - H_{t}}_{W^{1,p}(\S^n)}
		\leq
		\varepsilon\text{.}
	\]
\end{corollary}

\begin{proof}[Proof of Proposition~\ref{proposition:controlled_homotopies}]
	Let $ G \in C^{\infty}(\S^n\times[0,1],\n)$ be any homotopy connecting $ \varphi $ to $ \psi $ with $G_0 = \varphi$ and $G_1 = \psi$.
	Since $ \varphi = \psi $ outside of $ B_r(x)  $, we may assume that $ G $ is stationary outside of $ B_r(x)  $, i.e., for each $t\in[0,1]$, we have $G_t= \varphi = \psi$ on $\S^n \setminus B_r(x)$.
	This claim can be proved with a by-hand construction, that we sketch below.
	We denote by \( \hat{x} \) the point at the antipode of \( x \).
	Let \( \Psi \colon \S^{n} \to \S^{n} \) be a smooth map such that \( \Psi = \id \) outside of \( B_{r}(x) \), and such that \( \Psi \) maps the annulus \( B_{r}(x) \setminus \overline{B}_{r/2}(x) \) diffeomorphically onto the annulus \( \S^{n} \setminus (\overline{B}_{r}(x) \cup \{\hat{x}\}) \), the circle \( \partial B_{r/2}(x) \) onto \( \{\hat{x}\} \), and the ball \( B_{r/2}(x) \) diffeomorphically onto \( \S^{n} \setminus \{\hat{x}\} \).
	It is readily observed that \( \Psi \sim \id \), through a homotopy stationary outside of \( B_{r}(x) \).
	Therefore, the maps \( u \circ \Psi \) and \( v \circ \Psi \) are homotopic to \( u \) and \( v \) respectively, through a homotopy stationary outside of \( B_{r}(x) \).
	Now, given a homotopy \( G' \) connecting \( u \) to \( v \), a homotopy \( G'' \) connecting \( u \circ \Psi \) to \( v \circ \Psi \) can be constructed by prescribing that \( G'' \) is stationary outside of \( B_{r} \), by letting \( G''_{t} = G'_{t} \circ \Psi \) on \( \overline{B}_{r/2} \) --- which corresponds to rescaling \( G' \) from \( \S^{n} \setminus \{ \hat{x} \}\) to \( B_{r}(x) \) --- and extending smoothly on the annulus \( B_{r} \setminus \overline{B}_{r/2} \).
	This is readily done by combining the observations that (i) \( u \circ \Psi \) and \( v \circ \Psi \) coincide also on \( B_{r} \setminus B_{r/2}(x) \) and are constant on \( \partial B_{r/2}(x) \), and (ii) \( G''_{t} \) is constant on \( \partial B_{r/2}(x) \).
	The required homotopy \( G \) stationary outside of $ B_r(x)  $ is then obtained by patching the three above homotopies, from \( u \) to \( u \circ \Psi \), from \( u \circ \Psi \) to \( v \circ \Psi \), and from \( v \circ \Psi \) to \( v \).
	
	Consider $ \tau > 0 $, which will be chosen sufficiently small at a later stage.
	We are going to rescale $ G $, $ \varphi $, and $ \psi $ from $ B_r(x)  $ to a smaller ball $ B_\tau(x) $, while keeping them unchanged outside of $ B_{2r}(x) $.
	More specifically, let $ \brac{\Phi_{t}}_{0 \leq t \leq 1} $ be a family of smooth diffeomorphisms of $ \S^n $ such that $ \Phi_{t} = \id $ outside of $ B_{2r}(x) $ and such that, on $ B_{2r}(x) $, in the local chart given by the exponential map around $ x $, $ \Phi_{t} $ is expressed as
	\[
		\begin{cases}
			\frac{rx}{\brac{1-t}r+t\tau} & \text{if $ \abs{x} \leq \brac{1-t}r+t\tau $,} \\
			\frac{x}{\abs{x}}\brac{\frac{r}{2r-\brac{1-t}r-t\tau}\brac{\abs{x}-\brac{1-t}r-t\tau}+r} & \text{if $ \brac{1-t}r+t\tau \leq \abs{x} \leq 2r $.}
		\end{cases}
	\]
	We define $ H \in C^{\infty}(\S^n\times[0,1],\n)$ by
	\[
		H_{t} \coloneqq
		\begin{cases}
			\varphi \circ \Phi_{3t} & \text{if $ 0 \leq t \leq \frac{1}{3} $,} \\
			G_{3\brac{t-1/3}} \circ \Phi_{1} & \text{if $ \frac{1}{3} \leq t \leq \frac{2}{3} $,} \\
			\psi \circ \Phi_{1-3\brac{t-2/3}} & \text{if $ \frac{2}{3} \leq t \leq 1$.}
		\end{cases}
	\]
	Of course, $ H $ is a homotopy from $ \varphi $ to $ \psi $.
	It remains to show that, if $ \tau > 0 $ is suitably small, then $ H $ satisfies the required estimate.

	For $ 0 \leq t \leq \frac{1}{3} $, we note that $\varphi - H_t =0$ outside $B_{2r}(x)$. We readily obtain bounds on the Jacobian and the derivatives of $ \Phi_{t} $, so that the change of variable theorem combined with $n-p>0$ implies that
	\[
			\norm{\varphi - H_{t}}_{W^{1,p}(\S^n)}
		\leq
		\norm{\varphi}_{W^{1,p}(B_{2r}(x))} + \norm{\varphi \circ \Phi_{3t}}_{W^{1,p}(B_{2r}(x))}
		\aleq \norm{\varphi}_{W^{1,p}(B_{2r}(x))}\text{.}
	\]
	Similarly, for $ \frac{2}{3} \leq t \leq 1 $, we have
	\[
		\norm{\varphi - H_{t}}_{W^{1,p}(\S^n)}
		\leq
		\norm{\varphi}_{W^{1,p}(B_{2r}(x))} + \norm{\psi \circ \Phi_{3t}}_{W^{1,p}(B_{2r}(x))}
		\lesssim
		\norm{\varphi}_{W^{1,p}(B_{2r}(x))} + \norm{\psi}_{W^{1,p}(B_{2r}(x))}\text{.}
	\]
	Concerning $ \frac{1}{3} \leq t \leq \frac{2}{3} $, we estimate
	\[
	\begin{split}
		\norm{\varphi - H_{t}}_{W^{1,p}(\S^n)}
		&\leq
 		\norm{\varphi}_{W^{1,p}(B_{2r}(x))} + \norm{G_{3\brac{t-1/3}} \circ \Phi_{1}}_{W^{1,p}(B_{2r}(x))}\\
 		&\lesssim
  		\norm{\varphi}_{W^{1,p}(B_{2r}(x))} + \norm{G_{3\brac{t-1/3}}}_{W^{1,p}(B_{2r}(x) \setminus B_r(x)) } + \tau^{\frac{n-p}{p}}\norm{G_{3\brac{t-1/3}}}_{W^{1,p}(B_{2r}(x))}\text{.}
	\end{split}
	\]
	Since the homotopy $ G $ has been assumed to be stationary outside of $ B_r(x)  $, we know that $ \norm{G_{3\brac{t-1/3}}}_{W^{1,p}(B_{2r}(x) \setminus B_r(x))} = \norm{\varphi}_{W^{1,p}(B_{2r}(x)\setminus B_r(x) )} $.
	On the other hand, by compactness, we have
	\[
		\sup_{0 \leq t \leq 1} \norm{G_{t}}_{W^{1,p}(B_{2r}(x))}
		\leq
		C_1
	\]
	for some possibly large constant $ C_1 > 0 $.
	We may assume that either \( \norm{\varphi}_{W^{1,p}(B_{2r}(x))} \neq 0 \) or \( \norm{\psi}_{W^{1,p}(B_{2r}(x))} \neq 0 \).
	Indeed, if \( \norm{\varphi}_{W^{1,p}(B_{2r}(x))} = 0 = \norm{\psi}_{W^{1,p}(B_{2r}(x))} \), this implies that both \( \varphi \) and \( \psi \) are identically zero --- note that this may only happen if \( 0 \in \n \) --- and we may directly conclude by choosing \( H \) to be constantly zero.
	As $ p < n $, we may therefore choose $ \tau > 0 $ sufficiently small, depending on $ C_1 $, so that
	\[
		\tau^{\frac{n-p}{p}}\norm{G_{3\brac{t-1/3}}}_{W^{1,p}(B_{2r}(x))}
		\leq
		\norm{\varphi}_{W^{1,p}(B_{2r}(x))} + \norm{\psi}_{W^{1,p}(B_{2r}(x))}
		\quad
		\text{for every $ \frac{1}{3} \leq t \leq \frac{2}{3} $.}
	\]
	Hence, we deduce that
	\[
		\norm{\varphi - H_{t}}_{W^{1,p}(\S^n)}
		\lesssim
		\norm{\varphi}_{W^{1,p}(B_{2r}(x))} + \norm{\psi}_{W^{1,p}(B_{2r}(x))}
		\quad
		\text{for every $ \frac{1}{3} \leq t \leq \frac{2}{3} $.}
	\]
	This concludes the proof.
\end{proof}

In Corollary~\ref{corollary:small_homotopies}, both the $ \delta > 0 $ controlling $ \norm{\varphi-\psi}_{W^{1,p}(\S^n)} $ \emph{and} the $ r > 0 $ depend on $ \varepsilon $.
A very natural question is whether or not one may find a homotopy $ H $ so that $ \sup_{0 \leq t \leq 1} \norm{\varphi - H_{t}}_{W^{1,p}(\S^n)} $ is controlled only by $ \norm{\varphi-\psi}_{W^{1,p}(\S^n)} $.
More precisely, we formulate the following open question (cf. \cite[Problem, p.11]{MazowieckaStrzelecki}).

\begin{openproblem}
\label{open_problem:small_homotopy}
	Let $ \varphi \in C^{\infty}(\S^n,\n) $.
	Does there exist some $ r > 0 $, possibly depending on $ \varphi $, such that for every $ x \in \S^n $ and every $ \psi \in C^{\infty}(\S^n,\n) $ homotopic to $ \varphi $ and satisfying $ \varphi = \psi $ on $ \S^n \setminus B_r(x)  $, there exists a homotopy $ H \in C^{\infty}(\S^n\times[0,1],\n)$ from $ \varphi $ to $ \psi $ such that
	\[
		\sup_{0 \leq t \leq 1} \norm{\varphi - H_{t}}_{W^{1,p}(\S^n)}
		\leq
		\omega\brac{\norm{\varphi-\psi}_{W^{1,p}(\S^n)}}\text{,}
	\]
	where $ \omega $ is a modulus of continuity satisfying $ \omega\brac{t} \to 0 $ as $ t \to 0 $.
\end{openproblem}

One may expect $ \omega $ to be linear in $ t $, but any modulus of continuity would already be of interest.
The question is already interesting for maps $ \S^2 \to \S^2 $.

\section{Proof of the generic non-uniqueness}\label{s:proof}
\begin{proof}[Proof of \Cref{th:main}]
Fix $\eps>0$ and $\varphi \in C^\infty(\S^2,\S^2)$. 
We note first that, by \Cref{th:ALuniqueness} combined with H\"{o}lder's inequality, we may find another mapping $\varphi_0 \in C^{\infty}(\S^2,\S^2)$ which admits exactly one energy minimizer $u_0\colon B^3\to\S^2$ among all maps having boundary datum $\varphi_0$, and such that \(\varphi_{0} \) differs from $\varphi$ only on a set $B_\frac\eps2(x_0)$ for some $x_0\in\S^2$ and is such that
\begin{equation}\label{eq:phiphi0difference}
\norm{\varphi - \varphi_0}_{W^{1,p}(\S^2)} <\frac{\eps}{2}\text{.}
\end{equation}
We recall that, combining the regularity result \cite[Theorem II]{SU1} with the boundary regularity \cite[Theorem 2.7]{SU2} of Schoen--Uhlenbeck, $u_0$ can have only a finite number of singularities; let us denote this number by $M = \#\sing u$ (possibly $M=0$).

Next, we apply \Cref{corollary:small_homotopies} to $\varphi_0\in C^\infty(\S^2,\S^2)$. We obtain the existence of a $\delta=\delta(\eps)>0$ and an $r=r(\varphi_0,\eps)>0$ such that for any $\psi\in C^{\infty}(\S^2,\S^2)$ that differs from $\varphi_0$ only on the set $B_r(x_0)$ and such that $\norm{\varphi_0-\psi}_{W^{1,p}(\S^2)}<\delta$, there exists a homotopy $H\in C^\infty(\S^2 \times [0,1],\S^2)$ with
\begin{equation}\label{eq:specialhomotopy}
 \sup_{0\le t \le 1} \norm{\varphi_0 - H_t}_{W^{1,p}(\S^2)}<\frac \eps 2.
\end{equation}
Let $\eps_1 \coloneqq \min\{\delta,r,\frac \eps 2\}$. 
By \cite[Theorem 2.3.1]{MazowieckaPhD}, we construct $\varphi_1 \in C^{\infty}(\S^2,\S^2)$ with the properties:
\begin{enumerate}
 \item $\deg \varphi_0 = \deg \varphi_1$;
 \item $\|\varphi_0 - \varphi_1\|_{W^{1,p}}<\eps_1$ and $\varphi_0 = \varphi_1$ except on $B_{\eps_1}(x)$ for some point $x\in\S^2$;
 \item $\varphi_1$ admits only one energy minimizer $u_1\colon B^3 \to \S^2$ having at least $M + 1$ singularities.
\end{enumerate}
To be precise, the statement \cite[Theorem 2.3.1]{MazowieckaPhD} gives only that $\mathcal H^2(\{x\in\S^2\colon \varphi_0(x) \neq \varphi_1(x)\})<\eps_1$, but following the lines of the proof, we may deduce that $\varphi_0 = \varphi_1$ except on $B_{\eps_1}(x)$ for some point $x\in\S^2$.

Now, let us take the homotopy $H_t$ between $\varphi_0$ and $\varphi_1$ constructed in \Cref{corollary:small_homotopies}. Let
\[
\begin{split}
 \tau \coloneqq \sup\{t\in[0,1]\colon &\text{ each energy minimizer with boundary datum $H_t$}\\
 &\text{ has at most $M$ singular points in } B^3\}\text{.}
 \end{split}
\]
We argue like in \cite[Remark 4.1]{MazowieckaStrzelecki} (which is a modified argument from \cite[Section 5]{HardtLin89}).
For the convenience of the reader, we state here the main lines of the reasoning. First, we note that from the Stability Theorem \cite{HardtLin89}, see also \cite[Theorem 8.9]{MMS}, we have $\tau\in(0,1)$.

Now take $s_i \nearrow \tau$ and a sequence of minimizing harmonic maps $u_i\in W^{1,2}(B^3,\S^2)$ with $u_i\big\rvert_{\partial B^3} = H_{s_i}$ and $\# \sing u_i \le M$. 
Let us also take $t_i \searrow \tau$ and a sequence of minimizing harmonic maps $v_i\in W^{1,2}(B^3,\S^2)$ with $v_i\big\rvert_{\partial B^3} = H_{t_i}$ and $\# \sing v_i > M$. 
Since $\sup_i \brac{[H_{s_i}]_{W^{1,2}(\S^2)}+ [H_{t_i}]_{W^{1,2}(\S^2)}} <\infty$, we may deduce from the strong convergence of minimizers, see \cite[Theorem 1.2 (4)]{AL88} (see also \cite[Theorem 6.1 (3)]{MMS}), that up to a subsequence we have
\[
\begin{split}
 u_i &\to u \quad \text{ strongly in } W^{1,2}(B^3,\S^2)\text{,}\\
 v_i &\to v \quad \text{ strongly in } W^{1,2}(B^3,\S^2)\text{,}
\end{split}
 \]
and both $u$ and $v$ are energy minimizers with $u\big\rvert_{\partial B^3} = v\big\rvert_{\partial B^3} = H_\tau$. 
We claim that $\#\sing u \le M$. 
Indeed, assume on the contrary that $\# \sing u >M$. 
Then, by \cite[Theorem 1.8 (2)]{AL88} (see also \cite[Theorem 2.10]{MMS}), we would obtain that for each $y\in \sing u$ and for sufficiently large $i$, there would exist $y_i \in \sing u_i$ with $y_i \to y$ as $i\to \infty$, a contradiction.

Moreover, $\#\sing v > M$. 
To see this, let us again assume by contradiction that $\#\sing v \le  M$. 
Let now $z_{i,j}\in \sing v_i$ for $j\in\{1,\ldots,M+1\}$ be distinct singular points of $v_i$. 
Now let us observe that for sufficiently large $i$, we know that that $H_{t_i}$ and $H_\tau$ are close in $C^\infty$. 
Hence, by uniform boundary regularity~\cite[Theorem 1.10 (2)]{AL88} (see also \cite[Theorem 7.4]{MMS}), there is a uniform neighborhood of the boundary $\partial B^3$ which contains no singularities of $v$ and $v_i$, say $\dist(z,\partial B^3) \ge \lambda > 0$ for any $z\in \bigcup_i\sing v_i \cup \sing v$. 
Since singular points converge to singular points, we deduce from \cite[Theorem 1.8 (1)]{AL88} (see also \cite[Theorem 2.5]{MMS}) that for each $j$, we have $z_{i,j}\to z_j$ as $i\to \infty$ and $z_j\in\#\sing v$. 
The only possibility for $\# \{z_1,\ldots, z_{M+1}\}< M+1$ is that two singularities of $v_i$ converge to the same singularity of $v$. 
This, however, is impossible, because by the uniform distance between singularities \cite[Theorem 2.1]{AL88} (see also \cite[Theorem 2.12]{MMS}), there exists a universal constant $C$ (independent of the minimizer) such that no singularity can occur next to $z_{i,j}$ at a distance $C \,\dist(z_{i,j},\partial B^3) \ge C\lambda$.

Hence, $H_\tau \colon \S^2\to\S^2$ serves as a boundary condition for at least two minimizers $u$ and $v$ having a different number of singularities. Combining \eqref{eq:specialhomotopy} with \eqref{eq:phiphi0difference}, we obtain
\[
 \norm{\varphi - H_\tau}_{W^{1,p}(\S^2)} \le \norm{\varphi - \varphi_0}_{W^{1,p}(\S^2)} + \norm{\varphi_0 - H_\tau}_{W^{1,p}(\S^2)} < \frac{\eps}{2} + \eps_1 \leq \eps.
\]
This finishes the proof.
\end{proof}

\bibliographystyle{abbrv}%

\begin{thebibliography}{10}
	
	\bibitem{AL88}
	F.~J. Almgren, Jr. and E.~H. Lieb.
	\newblock Singularities of energy minimizing maps from the ball to the sphere:
	{E}xamples, counterexamples, and bounds.
	\newblock {\em Ann. of Math. (2)}, 128(3):483--530, 1988.
	
	\bibitem{BethuelZheng1988}
	F.~Bethuel and X.~Zheng.
	\newblock Density of smooth functions between two manifolds in {Sobolev}
	spaces.
	\newblock {\em J. Funct. Anal.}, 80(1):60--75, Sept. 1988.
	
	\bibitem{HardtKinderlehrer}
	R.~Hardt and D.~Kinderlehrer.
	\newblock Mathematical questions of liquid crystal theory.
	\newblock In {\em Nonlinear partial differential equations and their
		applications. {C}oll\`ege de {F}rance {S}eminar, {V}ol. {IX} ({P}aris,
		1985--1986)}, volume 181 of {\em Pitman Res. Notes Math. Ser.}, pages
	276--289. Longman Sci. Tech., Harlow, 1988.
	
	\bibitem{HKL-thevariety88}
	R.~Hardt, D.~Kinderlehrer, and F.-H. Lin.
	\newblock The variety of configurations of static liquid crystals.
	\newblock In {\em Variational methods ({P}aris, 1988)}, volume~4 of {\em Progr.
		Nonlinear Differential Equations Appl.}, pages 115--131. Birkh\"{a}user
	Boston, Boston, MA, 1990.
	
	\bibitem{HKLu}
	R.~Hardt, D.~Kinderlehrer, and M.~Luskin.
	\newblock Remarks about the mathematical theory of liquid crystals.
	\newblock In {\em Calculus of variations and partial differential equations
		({T}rento, 1986)}, volume 1340 of {\em Lecture Notes in Math.}, pages
	123--138. Springer, Berlin, 1988.
	
	\bibitem{HLp}
	R.~Hardt and F.-H. Lin.
	\newblock Mappings minimizing the {$L^p$} norm of the gradient.
	\newblock {\em Comm. Pure Appl. Math.}, 40(5):555--588, 1987.
	
	\bibitem{HardtLin89}
	R.~Hardt and F.-H. Lin.
	\newblock Stability of singularities of minimizing harmonic maps.
	\newblock {\em J. Differential Geom.}, 29(1):113--123, 1989.
	
	\bibitem{SiranLi}
	S.~Li.
	\newblock Stability of minimising harmonic maps under {$W^{1,p}$} perturbations
	of boundary data: {$p\geq2$}.
	\newblock {\em J. Differential Equations}, 296:279--298, 2021.
	
	\bibitem{MazowieckaPhD}
	K.~Mazowiecka.
	\newblock {\em Singularities of harmonic and biharmonic maps into compact
		manifolds}.
	\newblock PhD thesis, University of Warsaw, 2017.
	
	\bibitem{MMS}
	K.~Mazowiecka, M.~Mi\'{s}kiewicz, and A.~Schikorra.
	\newblock On the size of the singular set of minimizing harmonic maps.
	\newblock {\em Mem. Amer. Math. Soc.}, accepted.
	
	\bibitem{MazowieckaStrzelecki}
	K.~Mazowiecka and P.~Strzelecki.
	\newblock The {L}avrentiev gap phenomenon for harmonic maps into spheres holds
	on a dense set of zero degree boundary data.
	\newblock {\em Adv. Calc. Var.}, 10(3):303--314, 2017.
	
	\bibitem{RivierePhd}
	T.~Rivi{\`{e}}re.
	\newblock {\em Infinit\'{e}s des applications harmoniques \`{a} valeur dans une
		sph\`{e}re pour une condition au bord donn\'{e}e}.
	\newblock Published in Applications harmoniques entre vari\'{e}t\'{e}s.
	\newblock PhD thesis, Universit\'{e} Paris VI, 1993.
	
	\bibitem{SU1}
	R.~Schoen and K.~Uhlenbeck.
	\newblock A regularity theory for harmonic maps.
	\newblock {\em J. Differential Geom.}, 17(2):307--335, 1982.
	
	\bibitem{SU2}
	R.~Schoen and K.~Uhlenbeck.
	\newblock Boundary regularity and the {D}irichlet problem for harmonic maps.
	\newblock {\em J. Differential Geom.}, 18(2):253--268, 1983.
	
\end{thebibliography}

\end{document}